\documentclass[12pt,a4paper]{article}
\usepackage[utf8]{inputenc}
\usepackage{amsmath}
\usepackage{amsfonts}
\usepackage{amssymb}
\usepackage{amsthm}
\usepackage{hyperref}
\usepackage[left=2cm,right=2cm,top=2cm,bottom=2cm]{geometry}
\usepackage{xcolor}
\usepackage[all,cmtip]{xy}
\usepackage{tikz}
\usetikzlibrary{decorations.markings}
\usepackage{enumerate}

\newcommand{\RR}{\mathbb{R}}

\newcommand{\ZZ}{\mathbb{Z}}
\newcommand{\QQ}{\mathbb{Q}}

\newcommand{\NN}{\mathbb{N}}

\newcommand{\im}{\mathrm{im}}

\newtheorem{thm}{Theorem}[section]

\newtheorem{cor}[thm]{Corollary}
\newtheorem{lem}[thm]{Lemma}

\theoremstyle{definition}
\newtheorem{rem}[thm]{Remark}
\newtheorem{defn}[thm]{Definition}
\newtheorem{ex}[thm]{Example}

\begin{document}

\title{On integral Chang-Skjelbred computations with disconnected isotropy groups}
\author{Leopold Zoller}

\maketitle

\begin{abstract}
The Chang-Skjelbred method computes the cohomology of a suitable space with a torus action from its equivariant one-skeleton. We show that, under certain restrictions on the cohomological torsion, the integral cohomology is encoded in the one-skeleton even in the presence of arbitrary disconnected isotropy groups. We provide applications to Hamiltonian actions as well as to the GKM case. In the latter, our results lead to a modification of the GKM formula for graph cohomology.
\end{abstract}

\section{Introduction}
Let $T$ be a compact torus. For a suitably nice $T$-space, say a finite $T$-CW-complex, there is a powerful method to compute the (equivariant) cohomology of $X$ from the subspace $X_1$ which consists of all orbits of dimension $\leq 1$. Concretely, for some commutative coefficient ring $A$ (traditionally $A=\QQ$) and under the right topological assumptions (see below) the Chang-Skjelbred sequence
\begin{equation}\label{eq:CS}
0\rightarrow H^*_T(X;A)\rightarrow H^*_T(X^T;A)\rightarrow H^*_T(X_1,X^T;A)
\end{equation}
is exact. This tells us that $H^*_T(X;A)$ is embedded in $H^*_T(X^T;A)$ and that the image agrees with the image of the map on equivariant cohomology induced by the inclusion $X^T\rightarrow X_1$. Hence the ring $H^*_T(X;A)$ is completely encoded in $X_1$.

The appeal of this method is that the space $X_1$ is often much simpler than $X$ itself. A prominent setting where the method has been applied to great success is that of GKM manifolds (named after \cite{GKM}) where $X_1$ is a graph of $2$-spheres and $H^*(X;A)$ as well as $H^*_T(X;A)$ can be described completely in terms of the graph. This method of computation is also referred to as the ``GKM method''. While the origins of this approach, as developed by Chang and Skjelbred \cite{ChangSkjelbred}, use coefficients in a field of characteristic $0$, the method is also promising for integral coefficients. Integral aspects of the theory have also led to instances where in low dimensions, the combinatorics of $X_1$ indeed encode the (equivariant) homeomorphism/diffeomorphism type \cite{IsabelleLiat}, \cite{GKZCorrespondence}, extending classical results on toric varieties being determined by their combinatorics.

The question under which topological assumptions the sequence (\ref{eq:CS}) becomes exact is subtle but well studied. When passing from rational to integral coefficients, additional assumptions need to be placed to either restrain torsion in $H^*(X;\ZZ)$ or the occurring disconnected isotropy groups.
Concretely Sequence (\ref{eq:CS}) is known to be exact if
\begin{enumerate}
\item $A=\QQ$ and $H^*_T(X;\QQ)$ is free over $H^*(BT;\QQ)$. The latter phenomenon is often labelled equivariant formality and is equivalent to surjectivity of $H^*_T(X;\QQ)\rightarrow H^*(X;\QQ)$. Hence the nonequivariant cohomology $H^*(X)$ is encoded in $X_1$ as well. This is the original form as proved by Chang and Skjelbred \cite{ChangSkjelbred}. It was shown in \cite[Corollary 1.3]{AFPSyzygies} That the freeness assumption can be relaxed to assuming $H^*_T(X;\QQ)$ to be reflexive.
\item $A=\ZZ$, $H^*_T(X;\ZZ)\rightarrow H^*(X;\ZZ)$ is surjective and additionally all isotropy groups are connected \cite[Theorem 1.1]{FP:ECSWICFTA}. We note that the surjectivity is implied by freeness of $H^*_T(X;\ZZ)$ over $H^*(BT;\ZZ)$ but the converse holds only if $H^*(X;\ZZ)$ is torsion free, which is not necessarily the case here.
\item $A=\ZZ$, $H^*_T(X;\ZZ)$ is free over $H^*(BT;\ZZ)$ and additionally the isotropy of each point outside of $X_1$ is contained in a proper subtorus of $T$ \cite[Corollary 2.2]{FP:ESFEFS}. The assumption on $H^*(BT;\ZZ)$-freeness in particular forces $H^*(X)$ to be torsion free.
\end{enumerate}

Furthermore, exactness of Sequence (\ref{eq:CS}) in the context of certain other cohomology theories has been discussed in \cite{BaiPomerleano}. Our aim is to extend the range of possible applications for the above method to an integral setting without any restraints on the occurring isotropy groups as well as to sharpen the boundaries of this approach with computational examples and counterexamples. Denoting by $T_n\subset T$ the subgroup of all elements of order $n$, our main result is
\begin{thm}\label{thm:MainThmIntro}
Let $X$ be a finite $T$-CW complex such that $H^*_T(X;\ZZ)$ is free over $H^*(BT;\ZZ)$ and additionally $H^*(X^{T_n};\ZZ)$ is torsion free for all $n\in \NN$. Then the image of the injection $H^*_T(X;\ZZ)\rightarrow H^*_T(X^T;\ZZ)$ is encoded in $X_1$.
\end{thm}
In fact in our setup the Chang-Skjelbred sequence (\ref{eq:CS}) is not exact. The reason is that over $\ZZ$, in order to compute the image of $H^*_T(X)\rightarrow H^*_T(X^T)$, one needs to consider not only $X_1$ but also the fixed point sets $X^{T_p}$ of the maximal $p$-tori for all primes $p\in \ZZ$ (see Theorem \ref{thm:CS} and Remark \ref{rem:nonexact}). Our aim, with a view towards GKM applications, is to nonetheless compute $H^*_T(X)$ from $X_1$ by first analysing the $X^{T_p}$ via their one-skeleta. This leads to an inductive procedure which alternatingly applies the Chang-Skjelbred method to quotient groups $T/T_n$ and extends coefficient to account for ineffectivity. The detailed algorithm is described in Theorem \ref{thm:MainThm}.

Regarding the new technical condition of the $H^*(X^{T_n};\ZZ)$ being torsion free, we observe (Lemma \ref{lem:symplecticTorsion}) that for Hamiltonian actions it is does not place additional restrictions as it is contained in the condition of freeness of the equivariant cohomology. We note that this is equivalent to $H^*(M^T;\ZZ)$ being torsion free, so it applies in particular to Hamiltonian actions with discrete fixed points.

\begin{cor}\label{cor:mainHamiltonian}
Let $M$ be a compact Hamiltonian $T$-manifold such that $H^*_T(M;\ZZ)$ is free over $H(BT;\ZZ)$. Then the image of the injection $H^*_T(M;\ZZ)\rightarrow H^*_T(M^T;\ZZ)$ is encoded in $M_1$.
\end{cor}
On the other hand we show that, in general, the condition of $H^*_T(X;\ZZ)$ being free over $H^*(BT;\ZZ)$ is not sufficient to guarantee the success of the algorithm from Theorem \ref{thm:MainThm} without the additional condition of the $H^*(X^{T_n};\ZZ)$ being torsion free. We provide a counterexample in Example \ref{ex:counter}.

The rise of the popularity of the Chang-Skjelbred method partly comes from its computability in the case of a GKM manifold $M$. Under the right assumptions, $H^*_T(M;A)$ can be expressed via a combinatorial formula as the ``graph cohomology'' $H^*_T(\Gamma;A)$ of a labelled graph $\Gamma$ associated with the action, the so-called GKM graph (see Section \ref{sec:GKM}). In our setup $H^*_T(\Gamma;\ZZ)\cong H^*_T(M;\ZZ)$ no longer holds due to Sequence (\ref{eq:CS}) failing to be exact (see Example \ref{ex:notsimplify}). Our method translates into a modified version $\widehat{H}^*_T(\Gamma)$ of graph cohomology, which we define in Section \ref{sec:GKM}. Theorem \ref{thm:MainThmIntro} and Corollary \ref{cor:mainHamiltonian} then translate into
\begin{thm}
Let $M$ be an integer GKM manifold such that $H^*(M^{T_n};\ZZ)$ is torsion free for all $n\in \NN$ (which is automatic if the action is Hamiltonian). Then
\[H^*_T(M;\ZZ)\cong \widehat{H}^*_T(\Gamma).\]
\end{thm}
We note that while $H^*_T(\Gamma)$ is computed via a ``closed formula'', the definition of $\widehat{H}^*_T(\Gamma)$ relies on a finite number of recursive computations and does not seem to directly simplify into a closed expression (see the discussion in Example \ref{ex:notsimplify}).

The structure of the article is as follows: In Section $2$ we fix some notation and recall the integral version of the Chang-Skjelbred-Lemma which forms the basis of our considerations. Section $3$ is concerned with the proof of Theorem \ref{thm:MainThmIntro} as well as symplectic applications. These are applied to the GKM case in Section $4$. The final Section $5$ is concerned with the construction of counterexamples.

\paragraph*{Acknowledgements.} The author wants to thank Matthias Franz for helpful comments.

\section{Notation and preliminaries}

Throughout the paper $H^*(-)$ denotes singular cohomology. If no coefficients are specified they are taken in $\ZZ$.
For a topological group $G$ there is the universal $G$-Bundle $G\rightarrow EG\rightarrow BG$. For a $G$-space $X$ we associate the Borel fibration
\[X\rightarrow X_G\rightarrow BG\]
where $X_G$ is the quotient of $EG\times X$ by the diagonal $G$-action. The equivariant cohomology with coefficients in a ring $A$ is defined as $H^*_G(X;A)=H^*(X_G;A)$. The universal $G$-bundles may be constructed functorially in $G$ (see \cite{Milnor}) such that $H^*_G(X;A)$ becomes functorial in both $G$ and $X$. For us, the case of a compact torus $T$ is the most important. We set $R=H^*(BT)$. One has

\[R= \ZZ[x_1,\ldots,x_r]\]
where $r=\dim T$ and the $x_i$ are of degree $2$. We denote by $T_n\cong \ZZ_n^r$ the subgroup of $T$ of all elements of order $n$. For $n=p$ a prime we also call $T_p$ the maximal $p$-torus. The quotient $T/T_n$ is again a torus and we set $R_n=H^*(B(T/T_n))$. The projection $T\rightarrow T/T_n$ induces an injection $R_n\rightarrow R$ via which we identify $R_n$ as the subring
\[R_n= \ZZ[nx_1,\ldots,nx_r].\]

We will be making use of the classical Borel localization theorem for which we refer the reader to \cite[Theorem 3.1.6]{AP}. This is responsible for restricting to finite $T$-CW-complexes in the main results. We note that those assumptions can be somewhat weakened as is discussed in \cite[Section 3.2]{AP}. An important consequence of Borel localization is that when $X$ is a finite $T$-CW-complex and $H^*_T(X)$ is free over $R$ then the restriction $H^*_T(X)\rightarrow H^*_T(X^T)$ is injective. Its image is described by

\begin{thm}[Chang-Skjelbred]\label{thm:CS}
Let $X$ be a finite $T$-CW-complex and Assume that the image of $H^*_T(X)$ in $H^*_T(X^T)$ is free over $R$. Then
\[\im(H^*_T(X)\rightarrow H^*_T(X^T))=\bigcap_U \im (H^*_T(X^U)\rightarrow H^*_T(X^T))\]
where $U$ ranges over all codimension $1$ subtori as well as all maximal $p$-tori for all primes.
\end{thm}

The original version of Chang and Skjelbred \cite[Lemma 2.3]{ChangSkjelbred} used rational coefficients. The main difference to the integral version above is the occurrence of the maximal $p$-tori which are absent in the rational version (due to the fact that, unlike in the rational case, primes in $\ZZ$ have to be considered as meaningful divisors of elements in $H^2(BT;\ZZ)$). A proof of the above version can be found in \cite[Theorem 7.3.4]{AndersonFulton}, where it is formulated for varieties. The proof however relies only on the Borel localization theorem and works analogously to the proof of the rational version while additionally accounting for the primes in $\ZZ$.

\begin{rem}\label{rem:nonexact}
Since $X_1$ consists of all $X^U$ for all codimension $1$ subtori $U\subset T$ which are glued along $X^T$ an excision argument yields $H^*(X_1,X_T)\cong \bigoplus_{U\subset T} H^*(X^U,X^T)$. Consequently, the image of $H^*_T(X_1)\rightarrow H^*_T(X^T)$ is equal to the intersection in Theorem \ref{thm:CS} but with $U$ only running over codimension $1$ tori and not the maximal $p$-tori. It follows that, under the assumptions of Theorem \ref{thm:CS}, the Chang-Skjelbred-sequence (\ref{eq:CS}) is exact if and only if the maximal $p$-tori do not contribute to the intersection.
\end{rem}

\section{An inductive method}

Throughout this section $X$ is a finite $T$-CW-complex. The idea is to impose conditions under which the subspaces $X^{T_p}$ arising in the statement of Theorem \ref{thm:CS} do again satisfy the assumptions needed to apply Theorem \ref{thm:CS}. While the statement of Theorem \ref{thm:CS} is empty when applied to the ineffective $T$-action on $X^{T_p}$, this is no longer necessarily the case after passing to the $T/T_p$-action. Our assumptions will furthermore ensure that the passage between the $T$- and the $T/T_P$-action can be tracked on the level of cohomology. This results in a finite number of recursive computations which at every step only require information contained in the one-skeleton $X_1$. We begin by sorting out technicalities on the assumptions.

\begin{lem}\label{lem:bettinums}
Assume that $H^*_T(X)$ is free over $R$. Then the total Betti numbers of $X$ and any $X^{T_p}$ agree for any prime $p$.
\end{lem}

\begin{proof}
We have \[\dim_\QQ H^*(X;\QQ)=\dim_{\ZZ_p} H^*(X;\ZZ_p)\]
due to torsion freeness of $H^*(X)$.
The localization theorem applied to the $T_p$-action with $\ZZ_p$ coefficients implies that \[\dim_{\ZZ_p}H^*(X;\ZZ_p)\geq\dim_{\ZZ_p}H(X^{T_p};\ZZ_p)\] (see e.g.\ \cite[Theorem 3.10.4]{AP})
and the universal coefficient theorem yields \[\dim_{\ZZ_p} H^*(X^{T_p};\ZZ_p)\geq \dim_\QQ H^*(X^{T_p};\QQ).\]
An application of the localization theorem for the $T$-action on $X^{T_p}$ yields \[\dim_\QQ H^*(X^{T_p};\QQ)\geq \dim_\QQ H^*(X^T;\mathbb{Q})= \dim_\QQ H^*(X;\mathbb{Q})\] where the last equality follows from freeness of $H^*_T(X)$ over $R$. Hence \[\dim_\QQ H^*(X^{T_p};\QQ)= \dim_\QQ H^*(X;\mathbb{Q}).\]
\end{proof}

\begin{lem}\label{lem:free}
Assume that $H^*_T(X)$ is a free $R$-module and that for all $n\in \mathbb{N}$ the ordinary cohomology $H^*(X^{T_n})$ is torsion free. Then the $R_n$-module $H^*_{T/T_n}(X^{T_n})$ is free for any $n\in \mathbb{N}$.
\end{lem}

\begin{proof}
Assume first that $n=p$ is a prime. By assumption $H^*(X^{T_p})$ is torsion free so the collapse of the associated Serre spectral sequences of
\[X^{T_p}\rightarrow X^{T_p}_T\rightarrow BT\]
is equivalent over $\ZZ$ and $\QQ$. Furthermore for both coefficients this is equivalent to freeness of the total space cohomology over the base cohomology by the Leray--Hirsch theorem.

It  hence suffices to prove that the Serre spectral sequence of the Borel fibration of the $T/T_p$-action collapses rationally. As $(X^{T_p})^{T/T_p}=X^T$, this is equivalent to $\dim_\QQ H^*(X^{T_p};\QQ)=\dim_\QQ H^*(X^T;\QQ)$ (again see e.g.\ \cite[Theorem 3.10.4]{AP}). But by Lemma \ref{lem:bettinums} we have \[\dim_\QQ H^*(X^{T_p};\QQ)=\dim_\QQ H^*(X;\QQ)=\dim_\QQ H^*(X^T;\QQ).\]

Now for general $n$ we proceed via induction on the number of prime factors. Writing $n=mp$ for some prime $p$ we observe that $T/T_n$ agrees with the quotient of $T/T_m$ by its maximal $p$-torus $U$ and $X^{T_n}= (X^{T_m})^U$. Then the initial considerations for prime numbers yield the claim.
\end{proof}

The computation algorithm outlined in the beginning of this sections is made precise in the following theorem, where we assume $\dim T\geq 1$. The equivariant cohomologies of the $T$-fixed points $H_{T/T_n}^*(X^T)\cong R_n\otimes H^*(X^T)$ are understood as subalgebras of the common ambient algebra $H_T(X^T)\cong R\otimes H^*(X^T)$, where the intersection is taken.

\begin{thm}\label{thm:MainThm}
Assume that $H^*_T(X)$ is a free $R$-module and that for all $n\in \mathbb{N}$ the ordinary cohomology $H^*(X^{T_n})$ is torsion free. For $n\in \mathbb{N}$ and a subtorus $K\subset T$ set \[A_n=\im(H^*_{T/T_n}(X^{T_{n}})\rightarrow H^*_{T/T_n}(X^T)),\quad B_{n,K}= \im(H^*_{T/T_n}(X^{T_n\cdot K})\rightarrow H^*_{T/T_n}(X^T)).\]
Furthermore let $S$ be the finite set of $n$ for which $X^{T_n}\neq X^T$.
Then $H^*_T(X)\cong A_1$ can be recursively computed from the $B_{n,K}$ (and thus from $X_1$)
via
\[A_{n}=\left(\bigcap_{p} R_n\otimes_{R_{np}} A_{np}\right)\cap \left(\bigcap_K B_{n,K}\right)\]
where $p$ ranges over all primes with $pn\in S$ and $K$ ranges over all codimension $1$ subtori of $T$.
\end{thm}

\begin{proof}
The fact that $H^*_T(X)\cong A_1$ follows from the fact the restriction map $H^*_T(X)\rightarrow H^*_T(X^T)$ is injective by the localization theorem. For the description of higher $A_n$ we can apply Theorem \ref{thm:CS} to the $T/T_n$-action on $X^{T_n}$ whose requirements are fulfilled by Lemma \ref{lem:free}. The preimage of the maximal $p$-torus of $T/T_n$ under $T\rightarrow T/T_n$ is $T_{np}$. The preimages of a subtorus of $T/T_n$ is of the form $T_n\cdot K$ for some subtorus $K\subset T$. Thus we arrive at
\[A_n=\bigcap_U \im(H^*_{T/T_n}((X^U)\rightarrow H^*_{T/T_n}(X^T))\]
where $U$ ranges over $T_{np}$ for all primes and over all $T_n\cdot K$ where $K\subset T$ is a codimension $1$ subtorus. By the choice of $S$ we may disregard any $T_{np}$ with $np\notin S$. Consider the commutative diagram

\[\xymatrix{
H^*_{T/T_{np}}(X^{T_{np}})\ar[r]\ar[d] & H^*_{T/T_{np}}(X^T)\ar[d]\\
H^*_{T/T_{n}}(X^{T_{np}})\ar[r]& H^*_{T/T_{n}}(X^T)
}\]
induced by the projection $T/T_n\rightarrow T/T_{np}$ and the inclusion of the respective subspaces.
The $R_{np}$-module $H^*_{T/T_{np}}(X^{T^{np}})$ is free by Lemma \ref{lem:free} and the Serre spectral sequence of the corresponding Borel construction collapses. The pullback  of this fibration along $B(T/T_n)\rightarrow B(T/T_{np})$ is the Borel fibration of the $T/T_n$-action on $X^{T/T_{np}}$. Hence the corresponding spectral sequence collapses as well. Consequently $H^*_{T/T_n}(X^{T_{np}})$ is a free $R_n$-module and $H^*_{T/T_{np}}(X^{T_{np}})\rightarrow H^*_{T/T_{n}}(X^{T_{np}})$ maps an $R_{np}$-basis to an $R_n$-basis. It follows that \[R_n\otimes _{R_{np}} H^*_{T/T_{np}}(X^{T^{np}})\rightarrow H^*_{T/T_n}(X^{T_{np}})\] is an isomorphism. We deduce that $\im (H^*_{T/T_n}(X^{T_{np}})\rightarrow H^*_{T/T_n}(X^T))=R_n\otimes_{R_{np}} A_{np}$.
\end{proof}

Returning to the topic of investigating the necessary requirements needed for the inductive procedure above to work, we remark that the situation is particularly easy in the case of Hamiltonian actions. One has:

\begin{lem}\label{lem:symplecticTorsion}
Let $M$ be a compact Hamiltonian $T$-manifold, then the following are equivalent:
\begin{enumerate}[(i)]
\item The $R$-module $H^*_T(M)$ is free.
\item $H^*(M^U)$ is torsion free for some closed subgroup $U\subset T$.
\item $H^*(M^U)$ is torsion free for every closed subgroup $U\subset T$.
\end{enumerate}
\end{lem}

\begin{proof}
As argued in the proof of Lemma \ref{lem:free} torsion freeness of $H(M)$ implies that $H^*_T(M)$ is $R$-free if and only if $H^*_T(M;\QQ)$ is free over $H^*(BT;\QQ)$. By \cite{Frankel} in the Hamiltonian case $H^*_T(M;\QQ)$ is free over $H^*(BT;\QQ)$ and $H^*(M)$ has torsion if and only $H^*(M^T)$ has torsion (see the Theorem in Section 4 and Corollary 1 therein). While being formulated for Kähler manifolds in the reference, the result applies equally to the Hamiltonian setting. Now the equivalences in the Lemma follow from the fact, that the $M^U$ are all compact Hamiltonian $T$-manifolds sharing the same fixed point set.
\end{proof}
Combining this with Theorem \ref{thm:MainThm} yields Corollary \ref{cor:mainHamiltonian} as stated in the introduction.

\section{Computational aspects: the GKM case}\label{sec:GKM}

One class of spaces where Chang-Skjelbred computations are most prominently used is that of GKM manifolds.

\begin{defn}
A GKM manifold (resp.\ integer GKM manifold) is a compact orientable manifold $M$ with $H^*(M;\QQ)$ (resp. $H^*(M;\ZZ)$) concentrated in even degrees together with an action of a torus $T$ such that $M^T$ is finite and $M_1$ is a finite union of $2$-spheres.
\end{defn}

To a GKM manifold one can associate its GKM graph $(\Gamma,\alpha)$ where the vertices of $\Gamma$ are $V(\Gamma)=M^T$, the edges $E(\Gamma)$ are the invariant $2$-spheres (each containing $2$ fixed points) and $\alpha\colon E(\Gamma)\rightarrow H^2(BT;\ZZ)/\pm$ associates to each sphere the weight of the isotropy action at the fixed points (unique up to sign). For notational purposes we fix for each $e\in E(\Gamma)$ a starting vertex $i(e)$ and end vertex $t(e)$. The appeal of this class of spaces is that $M_1$ is completely encoded as a $T$-space in the combinatoric object $(\Gamma,\alpha)$. The Chang-Skjelbred method enables the computation of equivariant cohomolgy in terms of the graph under the right hypothesis and coefficients. Concretely, for a coefficient ring $A$ one defines the graph cohomology
\[H^*_T(\Gamma;A)=\{(f_p)_{p\in E(\Gamma)}\in H^*(BT;A)^{E(\Gamma)}\mid f_{i(e)}\equiv f_{t(e)}\mod \alpha(e)\text{ for all $e\in E(\Gamma)$}\}\]
where $\alpha(e)$ is interpreted in $H^2(BT;A)/\pm$. The Chang-Skjelbred Lemma, \cite[Lemma 2.3]{ChangSkjelbred}, \cite[Corollary 2.2]{FP:ESFEFS} then translates into the GKM theorem (cf. \cite{GKM})

\begin{thm}\label{thm:GKM} If $M$ is a GKM manifold and $A=\QQ$ or $M$ is an integer GKM manifold, $A=\ZZ$, and additionally all isotropy groups of points outside $M_1$ are contained in a proper subtorus, then
\[H^*_T(M;A)\cong H^*_T(\Gamma;A).\]
\end{thm}

\begin{rem}
One can show, see \cite[Proposition 3.10]{IsabelleLiat}, \cite[Proposition 2.8]{GKZZ2}, that the condition on the isotropy groups in the above integral case is equivalent to the condition that any two weights of adjacent edges are coprime in $H^2(BT;\ZZ)$ (withsome sign choice).
\end{rem}

Our Theorem \ref{thm:MainThm} translates into a computation algorithm that works without the assumption on the isotropy groups while adding an additional torsion freeness assumption on the cohomologies of certain submanifolds. For a GKM graph $(\Gamma,\alpha)$ and $n\in \NN$ we consider the labelled graph $(\Gamma_n,\alpha|_{E(\Gamma_n)})$ with $V(\Gamma_n)=V(\Gamma)$, $E(\Gamma_n)=\{e\in E(\Gamma)\mid \alpha(e)\equiv 0\mod n\}$. Note that the weights of $\Gamma_n$ lie in the image of $R_n\rightarrow R$, which we identify as a subring. We set
\[H^*_{T/T_n}(\Gamma_n)=\{(f_p)_{p\in V(\Gamma)}\in R_n^{V(\Gamma)}\mid f_{i(e)}\equiv f_{t(e)}\mod \alpha(e)\text{ for all $e\in E(\Gamma_n)$}\}\]
where the congruence condition is to be understood in $R_n^{V(\Gamma)}$.

\begin{defn}
For finite $\Gamma$ we recursively define
\[\widehat{H}^*_{T/T_n}(\Gamma_n)= {H^*}_{T/T_n}(\Gamma_n)\cap\bigcap_{p} R_n\otimes_{R_{np}}\widehat{H}^*_{T/T_{np}}(\Gamma_{np})\]
where $p$ ranges over all primes such that $np$ divides the weights of two distinct adjacent edges in $\Gamma$. For $n=1$ we also denote this by $\widehat{H}^*_T(\Gamma)$.
\end{defn}

Theorem \ref{thm:MainThm} now translates into the following extension of the GKM theorem.

\begin{thm}\label{thm:MainThmGKM}
Let $M$ be an integer GKM manifold and suppose that $H^*(M^{T_n})$ is torsion free for all $n\in \mathbb{N}$. Then 
\[H^*_T(M)\cong \widehat{H}^*_T(\Gamma).\]
\end{thm}

\begin{proof}
For each $n\in \mathbb{N}$ the space $M^{T_n}$ is again a disjoint union of integer GKM $T$-manifolds by Lemma \ref{lem:free}. The GKM graph is given by $(\Gamma_n,\alpha|_{E(\Gamma_n)})$. With the notation of Theorem \ref{thm:MainThm}, we show inductively that $A_n=\widehat{H}^*_{T/T_n}(\Gamma_n)$. If no adjacent weights are divisible by $np$ this is Theorem \ref{thm:GKM} applied to the $T/T_n$-action. Assume inductively that we have shown $A_{np}=\widehat{H}^*_{T/T_{np}}(\Gamma_{np})$ for all primes $p$. For a codimension $1$ subtorus $K\subset T$ the space $X^{T_nK}$ consists of discrete fixed points and copies of $S^2$ corresponding to those edges in $\Gamma_n$ whose weights vanish on $K$. Hence
\[\bigcap_K B_{n,K}=H^*_{T/T_n}(\Gamma_n)\] and Theorem \ref{thm:MainThm} gives the description
\[A_n=H^*_{T/T_n}(\Gamma_n)\cap\bigcap_p R_n\otimes_{R_{np}}\widehat{H}^*_{T/T_{np}}(\Gamma_{np})\]
with $p$ ranging over all primes. It remains to argue why it suffices to consider the intersection over those primes such that $np$ divides two adjacent edges. If $np$ does not do so, then $M^{T_{np}}$ is a disjoint union of copies of $S^2$ and discrete points. By the induction hypothesis $\widehat{H}^*_{T/T_{np}}(\Gamma_{np})=A_{np}$. The latter one computes to be generated by the following elements of $R_{np}^{V(\Gamma)}$: for every isolated vertex the element which is $1$ at that vertex and $0$ everywhere else and for every (isolated) edge the elements which are supported on the vertices belonging to the edge and take the values $(1,1)$ and $(0,\alpha)$, where $\alpha\in R_{np}^2$ is the weight of the edge.
In particular we observe that
\[R_n\otimes_{R_{np}}\widehat{H}^*_{T/T_{np}}(\Gamma_{np})\supset H^*_{T/T_n}(\Gamma_n)\]
does not contribute to the intersection. in the definition of $H^*_{T/T_n}(\Gamma_n)$.
\end{proof}

\begin{ex}\label{ex:notsimplify} We illustrate the method on an easy example.
Let $T=T^3$, $R=\ZZ[x_1,x_2,x_3]$, and $R_n=\ZZ[nx_1,nx_2,nx_3]$. Consider the $T$-action on $S^6$ which arises from the standard $T$-action by pulling back along the homomorphism $T\rightarrow T$, $(s,t,u)\mapsto (s^{p^2},t^{p^2},u^{pq})$ for distinct primes $p,q$. The GKM graph is

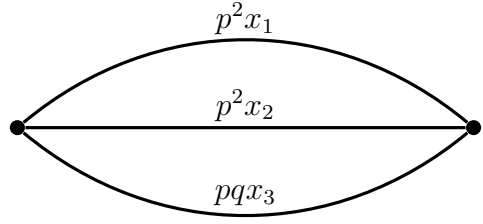
\begin{center}
		\begin{tikzpicture}

			\node (a) at (0,0)[circle,fill,inner sep=2pt] {};
			\node (b) at (6,0)[circle,fill,inner sep=2pt]{};
			\node at (3,.3) {$p^2x_2$};
			\node at (3,1.45) {$p^2x_1$};
	
			\node at (3,-0.9) {$pqx_3$};

			\draw [very thick](a) to[in=140, out=40] (b);
			\draw [very thick](a) to[in=180, out=0] (b);
			\draw [very thick](a) to[in=220, out=-40] (b);

		\end{tikzpicture}
	\end{center}
The example is not effective but can be easily turned into an effective one by adding more edges.
We denote $1=(1,1)\in R^2$ and $\alpha=(0,1)\in R^2$. In this case one easily sees via a pullback argument, that $H^*_T(S^6)=\langle 1, p^5qx_{123}\cdot\alpha\rangle_R$, where multiple indices denote the product of the respective $x_i$. We verify this using the method from Theorem \ref{thm:MainThmGKM}. The relevant graph cohomologies occurring in the recursion correspond to the common divisors of at least two adjacent edges. In particular $q,pq$ do not have to be considered. We compute
	\vspace*{0.3cm}
	\begin{center}

	\begin{tabular}{c|c|c}
	
	$H^*_T(\Gamma)$ & $H^*_{T/T_p}(\Gamma_p)$ & $H^*_{T/T_{p^2}}(\Gamma_{p^2})$ \\
	\hline
	$\langle 1,pqx_{123}\cdot\alpha\rangle_{R}$& $\langle 1,p^3qx_{123}\cdot\alpha\rangle_{R_p}$ & $\langle 1,p^4x_{12}\cdot\alpha\rangle_{R_{p^2}}$
	
	\end{tabular}
	\end{center}
	\vspace*{0.3cm}
	The rightmost entry already gives the respective value of $\widehat{H}^*$. Now we need to compute the intersection
	
	\[\widehat{H}^*_{T/T_p}(\Gamma_p)= H^*_{T/T_p}(\Gamma_p)\cap R_p\otimes_{R_{p^2}} H^*_{T/T_{p^2}}(\Gamma_{p^2}).\]
	The minimal coefficients in front of $x_{123}\alpha$ that one can produce in the respective factors of the intersection are (in the above order) given by $p^3q$ and $p^5$ coming from $(px_3)\cdot (p^4x_{12}\alpha)$. We stress the importance of the fact, that the factors are only $R_p$-modules, which forces the additional $p$-factor to be added to the second coefficient when compared to the generator in the table. Hence \[\widehat{H}^*_{T/T_p}(\Gamma_p)=\langle 1,p^5q x_{123}\cdot\alpha\rangle_{R_p}.\] These turn out to already be the $R$-generators for $\widehat{H}^*_T(\Gamma)$ since $H^*_T(\Gamma)\subset R\otimes_{R_p} \widehat{H}^*_{T/T_p}(\Gamma_p)$. We arrive at the desired result.
	
	This example illustrates that the recursive nature of the definition of $\widehat{H}^*_T$ is important: After computing the graph cohomologies in the table one cannot immediately extend all coefficients to $R$ and consider the intersections in $R^{V(\Gamma)}$ as this would yield the generators $1$, $p^4q x_{123}\cdot\alpha$. Instead, intersections and extensions of coefficients need to be performed in alternating fashion.
\end{ex}

\section{Counterexamples}
The goal of this section is to build an example of a $T$-action with $R$-free equivariant cohomology for which nevertheless the result of Theorem \ref{thm:MainThm} does not hold due to torsion appearing in some $H^*(X^{T_n})$. This shows that the latter condition is indeed necessary. We have divided the construction into several individual examples building upon another.

\begin{ex}[$T$-space $X$ with $R$-free $H^*_{T}(X)$ but $H^*(X^{T})$ has torsion]\label{ex:weirdsphere}

Set $T=S^1$. Consider the standard $S^1$-action on $\mathbb{R}P^2$ by viewing the latter as a disc with opposite boundary points identified and the action coming from the standard rotation. Now pull back this action along the homomorphism $s\mapsto s^3$ resulting in a $S^1$-space $A$ with principal isotropy $\mathbb{Z}_3$ and isotropy $\mathbb{Z}_6$ over the boundary of the disc. Now consider the $S^1$-space $B$ which arises from the standard rotation on $D^2$ by first dividing the boundary $\partial D^2$ by the standard $\mathbb{Z}_3$-action and then pulling back the induced $S^1$-action along the homomorphism $s\mapsto s^2$. The principal isotropy on $B$ is $\mathbb{Z}_2$ while the part coming from $\partial D^2$ is equivariantly homeomorphic to $S^1/\mathbb{Z}_6$. Now we form the space $X$ by gluing $A$ and $B$ along their respective copies of $S^1/\mathbb{Z}_6$, which we denote by $C$. We consider the Mayer-Vietoris sequence obtained by covering $X$ with $X\backslash C$ (which deformation retracts onto the fixed point set) and a small tube around $C$. The sequence becomes

\[0\rightarrow H_T(X)\rightarrow \mathbb{Z}[x]\oplus \mathbb{Z}[x]\oplus \mathbb{Z}[x]/(6x)\rightarrow \mathbb{Z}[x]/(3x)\oplus \mathbb{Z}[x]/(2x)\rightarrow 0\]
where the right hand map is given by $(a,b,c)\mapsto (a-c,b-c)$. Hence $H_{S^1}(X)$ is free. In fact $H^*_{S^1}(X)\cong H^*_{S^1}(S^2)$ for the standard rotation on $S^2$ and this isomorphism also holds for non-equivariant cohomology. However $X^{\mathbb{Z}_3}=A$ and $X^{\mathbb{Z}_2}=B$ have $\mathbb{Z}_2$ and $\mathbb{Z}_3$ torsion in their respective cohomology rings. We point out that despite this, the images in the fixed point cohomology are free.
\end{ex}

The next example is inspired by \cite[Example 5.2]{FP:ECSWICFTA} which carries a slightly different action and does not have $R$-free equivariant cohomology but has non-free image in the fixed point cohomology for the same reason.

\begin{ex}[$Y$ with $R$-free $H^*_{T}(Y)$ but $\im(H^*_{T/T_n}(Y^{T_n})\rightarrow H^*_{T/T_n}(Y^{T}))$ is not $R_n$-free]\label{ex:noice}
Set $T=S^1$.
Consider the space $X$ from Example \ref{ex:weirdsphere} and let $CX$ denote the cone over $X$. We form the space \[Y= (S^2\vee S^2)\cup (CX\times S^1)\] where the base layer $X\times S^1\subset CX\times S^1$ gets glued to $S^2\vee S^2$ along the map
$X\times S^1\rightarrow X\rightarrow X/C\cong S^2\vee S^2$. The $T$-action on $Y$ is defined by being the diagonal action on $CX\times S^1$, where the action on $CX$ is the one preserving the cone coordinate induced from the action on $X$ and the action on the right hand $S^1$ factor rotates with speed $3$. The gluing map then becomes equivariant with respect to the action on $S^2\vee S^2$ which rotates the first sphere at speed $3$ (the one which gets glued to $A$ in the terminology of Example \ref{ex:weirdsphere}) and the second one at speed $2$ (getting glued to $B$). A Mayer-Vietoris argument shows that
\[H^*(Y)^k\cong\begin{cases}
\mathbb{Z}& k=0,2,4\\
0&\text{else}
\end{cases}\]
The fixed point set consists of $3$ discrete points. Hence $H_T^*(Y)$ is $R$-free.
We wish to determine various images in the fixed point cohomology.

\begin{itemize}
\item $\im(H^*_T(Y)\rightarrow H^*_T(Y^T))$: \quad First note that the map $H^*(S^2\vee S^2)\rightarrow H^*(X)$ is surjective with kernel in $H^2(S^2\vee S^2)$ generated by the element $(2,3)$ w.r.t.\ the standard basis. This can be seen using cellular homology. One has $(X\times S^1)_T\cong X_{\mathbb{Z}_3}$ and a commutative diagram
\[\xymatrix{H^*_T(S^2\vee S^2)\ar[d]\ar[r] & H^*_T(X\times S^1)\ar[d]\\
H^*_{\ZZ_3}(S^2\vee S^2)\ar[r] & H^*_{\ZZ_3}(X)
}\]
in which the right hand vertical map is an isomorphism while the left hand vertical map is surjective. Now $H^*_{\ZZ_3}(X)=H^*(B\ZZ_3)\otimes H^*(X)$, $H^*_{\ZZ_3}(S^2\vee S^2)=H^*(B\ZZ_3)\otimes H^*(S^2\vee S^2)$ and the bottom horizontal map can be identified with the $H^*(B\ZZ_3)$-linear extension of the map on nonequivariant cohomology. In particular the upper horizontal map is surjective.

We obtain a Mayer-Vietoris sequence

\[0\rightarrow H^*_T(Y)\rightarrow H^*_T(S^2\vee S^2)\oplus H^*_T(CX\times S^1)\rightarrow H^*_T(X\times S^1)\rightarrow 0\]
in which the right hand map can be described as follows:

\[(\ZZ[x] \alpha\oplus \ZZ[x]\beta\oplus \ZZ[x] 1_{S^2\vee S^2})\oplus \ZZ[x]/(3x)  1_{S^1}\rightarrow \ZZ[x]/(3x) 1_X\oplus \ZZ[x]/(3x) \gamma\]
$\alpha\mapsto 3\gamma$, $\beta\mapsto -2\gamma$,
where $\alpha,\beta$ restrict to the standard generators of $H^2(S^2\vee S^2)$ and $\gamma$ restricts to a generator of $H^2(X)$. We conclude that the map $H^*_T(Y)\rightarrow H^*_T(S^2\vee S^2)$ has image the free $\ZZ[x]$-module generated by $1$, $2\alpha+3\beta$ and $3x(\alpha+\beta)$. We write elements in $H^*_T(Y^T)$ as triples in $H^*(BT)^3$ with the first two components corresponding to the fixed points in the first summand of $S^2\vee S^2$ (with principal isotropy $\ZZ_3$), the middle component corresponding to the shared fixed point, and the last component belonging to the fixed point which is exclusive to the second $S^2$ (with principal isotropy $\ZZ_2$). For the right choices of $\alpha,\beta$ we have $\alpha\mapsto (3x,0,0)$, $\beta\mapsto (0,0,2x)$. We find that $\im(H^*_T(Y)\rightarrow H^*_T(Y^T))$ is generated by $(1,1,1),(6x,0,6x),(9x^2,0,6x^2)$.

\item $\im(H^*_{T/T_3}(Y^{T_3})\rightarrow H^*_{T/T_3}(Y^T))$:\quad The space $Y^{T^3}$ is obtained by gluing $C\RR P^2\times S^1$ to $S^2\sqcup D^2$ via the map $\RR P^2\times S^1\rightarrow \RR P^2/\RR P^1=S^2$ on the bottom of the cone, as well as the inclusion $S^1\rightarrow D^2$ on the tip of the cone. We obtain a Mayer-Vietoris sequence
\[0\rightarrow H^*_{T/T_3}(Y^{T_3})\rightarrow H^*_{T/T_3}(S^2)\oplus H^*_{T/T_3}(D^2)\rightarrow H^*_{T/T_3}(\RR P^2\times S^1)\rightarrow 0.\]
Since the $T/T_3$-action on $\RR P^2\times S^1$ is free, one has $H^*_{T/T_3}(\RR P^2\times S^1)\cong H^*(\RR P^2)$. The right hand map in the above sequence is equivalent to
\[(\ZZ[3x]\cdot 1_{S^2}\oplus \ZZ[3x]\cdot \alpha)\oplus \ZZ[3x]\cdot 1_{D^2})\rightarrow \ZZ\cdot 1_{\RR P^2}\oplus \mathbb{Z}_2\cdot \delta\]
with $\alpha\mapsto \delta$, where $\alpha$ restricts to a generator of $H^2(S^2)$ and $\delta$ restricts to a generator in $H^2(\RR P^2)$. Hence the image of \[H^*_{T/T^3}(Y^{T_3})\rightarrow H^*_{T/T_3}(S^2)\oplus H^*_{T/T_3}(D^2)=H^*_{T/T_3}(S^2)\oplus H^*(B(T/T_3))\] is generated as an $\ZZ[3x]$-module by $(1,1),(0,3x),(2\alpha,0),(3x\alpha,0)$. In the fixed point cohomology $H^*_T(Y^T)$ this corresponds to the elements $(1,1,1),(0,0,3x),(6x,0,0),(9x^2,0,0)$.
\end{itemize}

\end{ex}

\begin{ex}[ $T$-space $Z$ with $R$-free $H^*_T(Z)$ where Theorem \ref{thm:MainThm} fails]\label{ex:counter}
We consider $Z=Y\times S^2$ where $Y$ is the $S^1$-space defined in Example \ref{ex:noice}. On $Z$, we consider the action of $T=T^2$ where the first factor acts on $Y$ and the second factor acts on the $S^2$ via effective rotation. We set $R=H^*(BT)=\ZZ[x,y]$ where $x$ is the generator induced by the projection onto the first circle factor and $y$ corresponds to the second. We observe that $Z_T=Y_{S^1}\times S^2_{S^1}$. Since both factors have torsion free cohomology, by the Künneth Theorem and by the considerations in Example \ref{ex:noice}, the image $H^*_T(Z)\rightarrow H^*_T(Z^T)\cong R^{2\times 3}$ is generated as an $R$-algebra by
\begin{align*}
\begin{pmatrix}
1&1&1\\1&1&1
\end{pmatrix},
\begin{pmatrix}
6x&0&6x\\6x&0&6x
\end{pmatrix},
\begin{pmatrix}
9x^2&0&6x^2\\9x^2&0&6x^2
\end{pmatrix},
\begin{pmatrix}
y&y&y\\0&0&0
\end{pmatrix},
\begin{pmatrix}
6xy&0&6xy\\0&0&0
\end{pmatrix},
\begin{pmatrix}
9x^2y&0&6x^2y\\0&0&0
\end{pmatrix}
\end{align*}
with each entry corresponding to a point in $Z^T=Y^{S^1}\times (S^2)^{S^1}$, where the different columns correspond to different fixed points of $Y^{S^1}$ enumerated as in Example \ref{ex:noice}. In particular the element
\[a=\begin{pmatrix}
3x^2y&0&0\\0&0&0
\end{pmatrix}\]
is contained in $H^*_T(Z)$. On the other hand $Z^{T_3}$ consists of $2$ disjoint copies of $Y$ and by the considerations in Example \ref{ex:noice}, the image $A_3$ of $H^*_{T/T_3}(Z^{T_3})\rightarrow H^*_{T/T_3}(Z^T)$ is generated as an $R_3=\ZZ[3x,3y]$-module by
\begin{align*}
&\begin{pmatrix}
1&1&1\\0&0&0
\end{pmatrix},
\begin{pmatrix}
6x&0&0\\0&0&0
\end{pmatrix},
\begin{pmatrix}
0&0&3x\\0&0&0
\end{pmatrix},
\begin{pmatrix}
9x^2&0&0\\
0&0&0
\end{pmatrix},\\
&\begin{pmatrix}
0&0&0\\1&1&1
\end{pmatrix},
\begin{pmatrix}
0&0&0\\6x&0&0
\end{pmatrix},
\begin{pmatrix}
0&0&0\\0&0&3x
\end{pmatrix},
\begin{pmatrix}
0&0&0\\9x^2&0&0
\end{pmatrix}
\end{align*}
We observe that $R\otimes_{R_3} A_3$ is generated as an $R$ module by the above elements and does not contain the element $a$.
\end{ex}

\bibliography{Realizationbib}
\bibliographystyle{acm}

\end{document}